\documentclass[11pt]{article}
\usepackage{a4wide}

\usepackage[a4paper, margin=2.3cm]{geometry}
\usepackage{amsmath,amssymb,amsthm}
\usepackage{color}
\usepackage{parskip}
\usepackage{hyperref}
\usepackage{parskip}
\usepackage{setspace}
\usepackage{graphicx}
\usepackage{mathtools}
\usepackage[thinc]{esdiff}
 
\newtheorem{theorem}{Theorem}[section]

\newtheorem{lemma}[theorem]{Lemma}

\newtheorem*{definition*}{Definition}

\def\C{\mathcal{C}}

\def\E{E}

\def \E{{\mathbb E}}
\def \B{{\mathcal B}}

\begin{document}
\title{Configurations of rectangles in a set in $\mathbb{F}_q^2$}
\author{
Doowon Koh\thanks{Department of Mathematics, Chungbuk National University, Korea. Email: koh131@chungbuk.ac.kr}\and
Sujin Lee\thanks{Department of Mathematics, Chungbuk National University, Korea. Email: sujin4432@chungbuk.ac.kr}\and
Thang Pham\thanks{Department of Mathematics, Vietnam National University. Email: phamanhthang.vnu@gmail.com}\and
Chun-Yen Shen\thanks{Department of Mathematics, National Taiwan University. Email: cyshen@math.ntu.edu.tw} }
\date{}
\maketitle
\begin{abstract}
Let $\mathbb{F}_q$ be a finite field of order $q$. In this paper, we study the distribution of rectangles in a given set in $\mathbb{F}_q^2$. More precisely, for any $0<\delta\le 1$, we prove that there exists an integer $q_0=q_0(\delta)$ with the following property: if $q\ge q_0$ and $A$ is a multiplicative subgroup of $\mathbb{F}^*_q$ with $|A|\ge q^{2/3}$, then  any set $S\subset \mathbb{F}_q^2$ with $|S|\ge \delta q^2$ contains at least $\gg \frac{|S|^4|A|^2}{q^5}$ rectangles with side-lengths in $A$. We also consider the case of rectangles with one fixed side-length and the other in a multiplicative subgroup $A$. 
\end{abstract}
\section{Introduction}
Let $\mathbb{F}_q$ be a finite field of order $q$, and we always assume that $q$ is a prime power with $q\equiv 3 \mod{4}.$  

A rectangle in $\mathbb{F}_q^d$ is a quadruple $(a, b, c, d)\in \mathbb{F}_q^d\times \mathbb{F}_q^d\times \mathbb{F}_q^d\times \mathbb{F}_q^d$ such that 
\[(a-b)\cdot (c-b)=0, ~(d-c)\cdot (b-c)=0, ~(a-d)\cdot (c-d)=0, ~(d-a)\cdot (b-a)=0 .\]

The main purpose of this paper is to study the following question: for $\lambda, \beta\in \mathbb{F}_q^*$ and $S\subset \mathbb{F}_q^2$, what conditions on $S$ will we need so that $S$ contains a rectangle with side-lengths $\alpha$ and $\beta$? Here the distance between two points $x$ and $y$ is defined by $||x-y||=(x_1-y_1)^2+(x_2-y_2)^2$. 

This question is similar to a question of  Graham in $\mathbb{Z}$: is it true that for any $\delta>0$, there exists $N_0=N_0(\delta)$ such that for any set $G$ in $[0, \ldots, N-1]\times [0, \ldots, N-1]$, $N>N_0$, of density at least $\delta$, one can find a quadruple of the form $\{(a, b), (a+d, b), (a, b+d), (a+d, b+d)\}$ in $G$ for some $d\ne 0$? 

An affirmative answer for this question can be found in \cite{soly, pre}, we refer the reader to \cite{shkredov, shkredov22} for more discussions.

In a recent paper, Lyall and Magyar \cite{lyall2} developed a counting lemma by using a weak hypergraph regularity lemma to prove that for any $0<\delta\le 1$, there exists an integer $q_0=q_0(\delta)$ such that if $q\ge q_0$ and $t_1, \ldots, t_d\in \mathbb{F}_q^*$, $d\ge 2$, then for any set $S\subset \mathbb{F}_q^{2d}$ with $|S|\ge \delta q^{2d}$ contains points 
\[\{x_{11}, x_{12}\}\times \cdots\times \{x_{d1}, x_{d2}\}\subset \mathbb{F}_q^{2d}, ~||x_{j2}-x_{j1}||=t_j, ~1\le j\le d.\]
So when $d=2$ and $S\subset \mathbb{F}_q^4$ with $|S|\ge \delta q^{4}$, and $q$ is large enough, we always can find a rectangle in $S$ with a given pair of non-zero side-lengths. 

Throughout this paper, we use the following notations: $X \ll Y$ means that there exists  some absolute constant $C_1>0$ such that $X \leq C_1Y$, and  $X\sim Y$ means $Y\ll X\ll Y$.

By using Lyall and Magyar's method and structures of multiplicative subgroups in $\mathbb{F}_q$, we are able to give some answers for two dimensional problem as follows. \\

\begin{theorem}\label{thm2}
For any $0<\delta\le 1$, there exists an integer $q_0=q_0(\delta)$ with the following property: if $q\ge q_0$ and $A$ is a multiplicative subgroup of $\mathbb{F}^*_q$ with $|A|\ge q^{2/3}$, then  any set $S\subset \mathbb{F}_q^2$ with $|S|\ge \delta q^2$ contains at least $\gg \frac{|S|^4|A|^2}{q^5}$ rectangles with side-lengths in $A$. 

\end{theorem}

\bigskip

\begin{theorem}\label{thm33333}

Let $t$ be a non-zero square number in $\mathbb{F}_q$. For any $0<\delta\le 1$, there exists an integer $q_0=q_0(\delta)$ with the following property: if $q\ge q_0$ and $A$ is a multiplicative subgroup of $\mathbb{F}^*_q$ with $|A|\ge q^{2/3}$, then  any set $S\subset \mathbb{F}_q^2$ with $|S|\ge \delta q^2$ contains at least $\gg \frac{|S|^4|A|}{q^5}$ rectangles with one side-length $t$ and the other side-length in $A$. 
\end{theorem}




There is also a series of papers which study similar geometric questions, we refer the interested reader to \cite{0, 1, 1.5, 2} for more details. 
\section{Proof of Theorem \ref{thm2}}
To prove Theorem \ref{thm2}, as mentioned in the introduction, we adapt the method developed in \cite{lyall2}. 

Let $A$ be a multiplicative subgroup of  $\mathbb{F}_q^*$ and $\chi$ denote a non-trivial additive character of $\mathbb F_q.$ We first give a proof for the case $|A|\ge q^{2/3}$ in an arbitrary finite field $\mathbb{F}_q$. Define $\sigma(x):=\frac{q}{|A|}$ if $x\in A$ and $0$ otherwise. We will invoke a classical result that if $|A|\ge q^{2/3}$, then 
\begin{equation}\label{MSExp}\left\vert \sum_{x\in A}\chi(x\cdot m)\right\vert\ll q^{1/2},\end{equation}
whenever $m\ne 0$ (for example, see \cite{PK}).

For functions $f_1, f_2, f_3, f_4\colon \mathbb{F}_q^2\to [-1,1]$, we define
\[N(f_1, f_2, f_3, f_4):=\E_{a, b, c, d}f_1(a, c)f_2(a, d)f_3(b, c)f_4(b, d)\sigma(a-b)\sigma(c-d), \]
and 
\[M(f_1, f_2, f_3, f_4):=\E_{a, b, c, d}f_1(a, c)f_2(a, d)f_3(b, c)f_4(b, d),\]
where $\E_{a,b,c,d}:= q^{-4} \sum_{a,b,c,d\in \mathbb F_q}.$
Recall that  we identify a set $S$ with the characteristic function $\chi_S$ on the set $S.$ For example, if $S\subset \mathbb F_q^2,$ then we have
\[N(S, S, S.S)=\E_{a, b, c, d}S(a, c) S(a, d)S(b, c) S(b, d) \sigma(a-b)\sigma(c-d),\]
and 
\[M(S, S, S, S)=\E_{a, b, c, d}S(a, c)S(a, d)S(b, c)S(b, d).\]
Using the Cauchy-Schwarz inequality two times for $|S|=\sum_{a}\sum_{b}S(a, b)$, we have
\begin{equation}\label{xxxx}M(S, S, S, S)\ge \left(\frac{|S|}{q^2}\right)^4.\end{equation}

Let $V_1$ and $V_2$ be pairwise orthogonal one-dimensional subspaces. We write $\mathbb F_q^2=V_1\times V_2$ with $V_j\simeq \mathbb F_q. $ 
For any function $f\colon\mathbb{F}_q^2\to [-1, 1],$  we define 
\[||f||_{\square(V_1\times V_2)}:=M(f, f, f, f)^{1/4}.\]
\begin{lemma}\label{basic1}
For functions $f_1, f_2, f_3, f_4\colon \mathbb{F}_q^2\to [-1, 1]$, we have 
\[ M(f_1, f_2, f_3, f_4)\le \min_{i}||f_i||_{\square(V_1\times V_2)}.\]
\end{lemma}
\begin{proof}
Applying the Cauchy-Schwarz inequality, we have 
\begin{align*}
M(f_1, f_2, f_3, f_4)&=\frac{1}{q^4}\sum_{a, b, c, d}f_1(a, c)f_2(a, d)f_3(b, c)f_4(b, d)\\
&=\frac{1}{q^4}\sum_{a, b}\left(\sum_{c}f_1(a, c)f_3(b, c)\right)\left(\sum_{d}f_2(a, d)f_4(b, d)\right)\\
&=\frac{1}{q^4}\left(  \sum_{a, b}\left(\sum_{c}f_1(a, c)f_3(b, c)\right)^{2}    \right)^{1/2}\cdot \left(  \sum_{a, b}\left(\sum_{d}f_2(a, d)f_4(b, d)\right)^{2}    \right)^{1/2}\\
&=M(f_1, f_1, f_3, f_3)^{1/2}\cdot M(f_2, f_2, f_4, f_4)^{1/2}.
\end{align*}
Repeating the argument for $M(f_1, f_2, f_3, f_3)$ and $M(f_2, f_2, f_4, f_4)$, and using the fact that $M(f_i, f_i, f_i, f_i)\le 1$ for $1\le i\le 4$, the lemma  follows. 
\end{proof}
To prove Theorem \ref{thm2}, we need two key lemmas, the first one is a weak hypergraph regularity lemma, and the second is a generalized von-Neumann type estimate. 
\subsection{A weak hypergraph regularity lemma}
Let $V$ be a subset in $\mathbb{F}_q$. Let $\B$ be a $\sigma$-algebra on $V$, i.e. a collection of sets in $V$ which contains $V$, $\emptyset$, and is closed under finite intersections, unions, and complementation. One can think of $\B$ as a partition of $V$, but the $\sigma$-algebra term gives us a better understanding since it offers intuitive ideas from measure theory and probability theory. 

For a $\sigma$-algebra $\B$ on $V$, the complexity of $\B$, denoted by $\mathtt{complexity}(\B)$, is defined to be the smallest number of sets in $V$ needed to generate $\B$. Each of those sets will be called an atom of $
\B$. 

Let $f\colon V\to \mathbb{R}$ be a function, we define the conditional expectation $\E(f|\B)\colon V\to \mathbb{R}$ by the formula 
\[\E(f|\B)(x):=\frac{1}{|\B(x)|}\sum_{y\in \B(x)}f(y),\]
where $\B(x)$ denotes the smallest element of $\B$ that contains $x$. 

If $\B_1$ and $\B_2$ are two $\sigma$-algebras on $V_1$ and $V_2$, respectively, we use the notation $\B_1\vee\B_2$ to denote the smallest $\sigma$-algebra on $V_1\times V_2$ that contains both $\B_1\times V_2$ and $V_1\times \B_2$. Note that atoms of $\B_1\vee\B_2$ are sets of the form $U\times V$, where $U$ and $V$ are atoms of $\B_1$ and  $\B_2$, respectively.

The following is a consequence of Lemma 2.2 in \cite{lyall2}. Notice that the authors in \cite{lyall2}  assumed that $V_1\simeq \mathbb F_q^2\simeq V_2$, but the argument is identical with $V_1=V_2=\mathbb{F}_q$ and with the corresponding norm $\square(V_1\times V_2)$.
\begin{lemma}\label{co2.2}
Let $V_1=V_2=\mathbb{F}_q$ and $S\subset V_1\times V_2$. For any $\epsilon>0$, there exist $\sigma$-algebras $\B$ on $V_1$ and $\C$ on $V_2$ such that each algebra is spanned by at most $O(\epsilon^{-8})$ sets, and 
 \[||S-\E(S| \B\vee\C)||_{\square(V_1\times V_2)}\le \epsilon.\]
\end{lemma}
We recall that
 \[\E(S|\B\vee\C)(x)=\frac{|S\cap B\times C|}{|B||C|},\]
where $B\times C$ is the atom of $\B\times \C$ containing $x$.

\subsection{A generalized von-Neumann type estimate}
We begin by recalling some notations and definitions in Fourier analysis over finite fields. Let $\chi$ denote a non-trivial additive character of $\mathbb F_q.$ For a function $f:\mathbb F_q^n\to \mathbb C,$  its Fourier transform, denoted by $\widehat{f}$, is defined by
$$ \widehat{f}(m)= q^{-n}\sum_{x\in \mathbb F_q^n} \chi(-m\cdot x) f(x).$$ 
For instance, we have
$$\widehat{\sigma}(m)=q^{-1}\sum_{m\in \mathbb F_q} \chi(-m\cdot x) \sigma(x)=|A|^{-1} \sum_{x\in A} \chi(-m\cdot x),$$
where we recall that for a multiplicative subgroup $A$ of $\mathbb F_q^*,$  we define $\sigma{(x)}=q/|A| $ for $x\in A,$ and $0$ otherwise. By \eqref{MSExp}, we obtain that if $|A|\ge q^{2/3},$ then
$$ \max_{m\ne 0}|\widehat{\sigma}(m)| \ll \frac{q^{1/2}}{|A|} \quad \mbox{and}~~ \widehat{\sigma}(0)=1.$$

By the orthogonality  of $\chi$,  one can easily deduce the following Fourier inversion formula: 
$$ \|\widehat{f}\|_2:= \left(\sum_{m\in \mathbb F_q^n} |\widehat{f}(m)|^2\right)^{1/2}= \left(q^{-n}\sum_{x\in \mathbb F_q^n} |f(x)|^2\right)^{1/2}.$$
We also have the Fourier inversion theorem:
$$ f(x)=\sum_{m\in \mathbb F_q^n} \widehat{f}(m) \chi(x\cdot m).$$ 
It is clear that $||\widehat{f}||_2 \le 1 $ for any function $f: \mathbb F_q \to [-1, 1].$
Now, we deduce a result on a generalized von-Neumann type estimate.
\begin{lemma}\label{co2.3}
For functions $f_1, f_2, f_3, f_4\colon \mathbb{F}_q^2\to [-1, 1]$, we have 
\[|N(f_1, f_2, f_3, f_4)|\le \min_{j} ||f_j||_{\square(V_1\times V_2)}  
+O\left(\frac{q^{1/8}}{|A|^{1/4}}\right),\]
where $A$ is a multiplicative subgroup of $\mathbb F_q^*$ with $|A|\ge q^{2/3}.$
\end{lemma}
\begin{proof}
First, for any one variable functions $f, g : \mathbb F_q \rightarrow [-1, 1]$, we proceed with estimating $\E_{x, y}f(x)g(y)\sigma(x-y).$  Applying the Fourier inversion formula to the function $\sigma(x-y)$, 
\begin{align*}
\E_{x, y}f(x)g(y)\sigma(x-y)&
=\frac{1}{q^2}\sum_{x, y}f(x)g(y)\sum_{m}\widehat{\sigma}(m)\chi(m\cdot (x-y))\\
&=\frac{1}{q^2}\sum_{x, y}f(x)g(y)+\frac{1}{q^2}\sum_{x, y}\sum_{m\ne 0}f(x)g(y)\widehat{\sigma}(m)\chi(m(x-y)),
\end{align*}
where we used a simple observation that $\widehat{\sigma}(0)=1.$
Thus, 
\begin{align*}
\left\vert \E_{x, y}f(x)g(y)\sigma(x-y)-\frac{1}{q^2}\sum_{x, y}f(x)g(y)\right\vert&=\left\vert\sum_{m\ne 0}\widehat{\sigma}(m) \overline{\widehat{f}}(m) \widehat{g}(m)\right\vert\\
&\le \max_{m\ne 0}|\widehat{\sigma}(m)|\cdot||\widehat{f}||_{L^2}||\widehat{g}||_{L^2}\ll \frac{q^{1/2}}{|A|}.
\end{align*}
This gives that 
\begin{equation}\label{vtK}\E_{x, y}f(x)g(y)\sigma(x-y)=\E_{x, y}f(x)g(y)+ O\left(\frac{q^{1/2}}{|A|}\right).\end{equation}
Since $f: \mathbb F_q\to [-1,1]$, we see that $|\E_{x, y}f(x)g(y)|\le 1.$ 
In addition, notice from our assumption, $|A|\ge q^{2/3}$, that $\frac{q^{1/2}}{|A|} \ge \left(\frac{q^{1/2}}{|A|}\right)^2.$
Then we have 
\[|\E_{x, y}f(x)g(y)\sigma(x-y)|^2=E_{x, y, z, t}f(x)g(z)f(y)g(t)+O\left(\frac{q^{1/2}}{|A|}\right) \le \E_{x, y}f(x)f(y)+O\left(\frac{q^{1/2}}{|A|}\right),\] where we have used $\E_{z,t}g(z)g(t) \leq 1$. 
In other words, for functions $f,g: \mathbb F_q \to [-1,1],$ we have
\begin{equation}\label{eq3}
|\E_{x, y}f(x)g(y)\sigma(x-y)|^2 \le  \E_{x, y}f(x)f(y)+ O\left(\frac{q^{1/2}}{|A|}\right).
\end{equation}
Notice that we also have
\begin{equation}\label{eq4}
|\E_{x, y}f(x)g(y)\sigma(x-y)|^2 \le  \E_{x, y}g(x)g(y)+ O\left(\frac{q^{1/2}}{|A|}\right).
\end{equation}
We now show that $N(f_1, f_2, f_3, f_4)\le ||f_1||_{\square(V_1\times V_2) }+O\left(\frac{q^{1/2}}{|A|}\right)$. The same argument works also for functions $f_2, f_3, f_4$. We recall that
\begin{align*}
N(f_1, f_2, f_3, f_4)&=\E_{a, b, c, d}f_1(a, c)f_2(a, d)f_3(b, c)f_4(b, d)\sigma(a-b)\sigma(c-d).
\end{align*}
For a fixed pair $(c, d)$, set $f_{c,d}(a)=f_1(a, c)f_2(a, d)$ and $g_{c,d}(b)=f_3(b, c)f_4(b, d).$ Applying the Cauchy-Swarz inequality, we have
 \begin{align*}
|N(f_1, f_2, f_3, f_4)|^2& \le \left(\E_{c, d}\sigma(c-d)|\E_{a, b}f_{c,d}(a)g_{c,d}(b)\sigma(a-b)|\right)^2\\
&\le \left(\E_{c,d}\sigma(c-d) \right) \left(\E_{c,d} \sigma(c-d) | \E_{a,b} f_{c,d}(a)g_{c,d}(b) \sigma(a-b)|^2\right)\\
& =\E_{c,d} \sigma(c-d) | \E_{a,b} f_{c,d}(a)g_{c,d}(b) \sigma(a-b)|^2,
\end{align*}
where we have used the observation that $\E_{c, d}\sigma(c-d)=1$ to get the last line.
Therefore, using the inequality \eqref{eq3}, we get 
\begin{align} \label{eq5K}|N(f_1, f_2, f_3, f_4)|^2&\le \E_{a,b,c,d} \sigma(c-d) f_{c,d}(a)f_{c,d}(b) +  \E_{c,d} \sigma(c-d)O\left(\frac{q^{1/2}}{|A|}\right) \nonumber \\ 
&=\E_{a, b, c, d}f_1(a, c)f_1(b, c)f_2(a, d)f_2(b, d)\sigma(c-d)+O\left(\frac{q^{1/2}}{|A|}\right).\end{align}
We repeat the argument one more time with $f_{a,b}(c)=f_1(a, c)f_1(b, c)$ and $g_{a,b}(d)=f_2(a, d)f_2(b, d)$ for each fixed pair $(a, b)$, as a consequence, we derive
$$ |N(f_1, f_2, f_3, f_4)|^4\le \E_{a,b,c,d} f_{a,b}(c) f_{a,b}(d) + O\left(\frac{q^{1/2}}{|A|}\right)
=M(f_1,f_1,f_1,f_1)+ O\left(\frac{q^{1/2}}{|A|}\right).$$
Hence, 
\[|N(f_1, f_2, f_3, f_4)|\le \min_{j}||f_j||_{\square (V_1\times V_2)}+O\left(\frac{q^{1/8}}{|A|^{1/4}}\right).\]
\end{proof}
\subsection{Proof of Theorem \ref{thm2}}
We are ready to give the proof of Theorem \ref{thm2}.
It follows from Lemma \ref{co2.2} that there exist $\sigma$-algebras $\B$ and $\C$ on $V_1$ and $V_2$ of complexity at most $O(\epsilon^{-8})$, respectively, such that 
 \begin{equation}\label{88}||S-\E(S| \B\vee\C)||_{\square(V_1\times V_2)}\le \epsilon.\end{equation}

Let $g=\E(S|\B\vee\C)$ and write $S(x)=g(x)+h(x)$. It follows from (\ref{88}) that 
\begin{equation}\label{epK}||h||_{\square(V_1\times V_2)}\le \epsilon.\end{equation}
 
Notice that $g$ and $h$ are functions from $\mathbb{F}_q^2\to [-1,1]$. 

Under the expression $S(x)=g(x)+h(x)$, we obtain 
\begin{align*}
N(S, S, S, S)=&\E_{a, b, c, d}S(a, c)S(a, d)S(b, c)S(b, d)\sigma(a-b)\sigma(c-d)\\
=&N(g, g, g, g)+N(h, h ,h, h)\\&+\sum_{(j_1,j_2,j_3, j_4)\in \Omega}\E_{a, b, c, d}F_{j_1}(a, c)F_{j_2}(a, d)F_{j_3}(b, c)F_{j_4}(b, d)\sigma(a-b)\sigma(c-d),\end{align*}
where $\Omega:=\{1,2\}^4\setminus \{(1,1,1,1), (2,2,2,2)\},$  $F_1:=g$, and $F_2:=h.$ Notice that if $(j_1, j_2, j_3, j_4)\in \Omega,$  then  $F_{j_i}=h$ for some $i=1,2,3,4.$

It follows from Lemma \ref{co2.3} and \eqref{epK} that 
\[|N(h, h, h, h)|\le ||h||_{\square (V_1\times V_2)}+ O\left(\frac{q^{1/8}}{|A|^{1/4}}\right)\le \epsilon+O\left(\frac{q^{1/8}}{|A|^{1/4}}\right),\]
and, for all $(j_1, j_2, j_3, j_4)\in \Omega$,  
$$|N(F_{j_1}, F_{j_2}, F_{j_3}, F_{j_4})| \le ||h||_{\square (V_1\times V_2)}+ O\left(\frac{q^{1/8}}{|A|^{1/4}}\right)\le \epsilon+O\left(\frac{q^{1/8}}{|A|^{1/4}}\right).$$

In other words, we now have that 
\begin{equation}\label{mot1}|N(S, S, S, S)-N(g, g, g, g)|=O(\epsilon)+O\left(\frac{q^{1/8}}{|A|^{1/4}}\right).\end{equation}

We know from Lemma \ref{basic1} that
$$|M(f_1, f_2, f_3, f_4)|\le \min_{j}||f_j||_{\square(V_1\times V_2)}.$$  
Thus, 
\begin{equation}\label{hai1}|M(S, S, S, S)-M(g, g, g, g)|=O(\epsilon).\end{equation}
It follows from the definition of the function $g$ that it is a linear combination of the indicator functions $1_{\B_i\times \C_j}$ of the atoms $\B_i\times \C_j$ of the $\sigma$-algebra $\B\vee\C$. We also know from Lemma \ref{co2.2} that the number of terms in this linear combination is at most $2^{C\epsilon^{-8}}$, with coefficient at most one in modulus, so the expression of $N(g, g, g, g)$ can be rewrite as a linear combination of terms of the form 
\[N(B\times C, E\times F, G\times H, U\times W)=\E_{a, b, c, d} 1_{B\times C}(a, c)1_{E\times F}(a, d)1_{G\times H}(b, c)1_{U\times W}(b, d)\sigma(a-b)\sigma(c-d),\]
for some atoms $B\times C, E\times F, G\times H, U\times W$ in $\B\times \C$. On the other hand, 
\begin{align*}
&N(B\times C, E\times F, G\times H, U\times W)\\
&=\E_{a}(B\cap E)(a)\E_{b}(G\cap U)(b)\E_{c}(C\cap H)(c)\E_{d}(F\cap W)(d)\sigma(a-b)\sigma(c-d)\\
&=\left[\E_{a,b} (B\cap E)(a)(G\cap U)(b)\sigma(a-b)\right] \left[\E_{c,d} (C\cap H)(c)(F\cap W)(d)
\sigma(c-d)  \right].
\end{align*}
By applying the estimate \eqref{vtK}, we get
\begin{align*} &N(B\times C, E\times F, G\times H, U\times W)\\
&=\E_{a}(B\cap E)(a)\E_{b}(G\cap U)(b)\E_{c}(C\cap H)(c)\E_{d}(F\cap W)(d)+ O\left(\frac{q^{1/2}}{|A|}\right)\\
&=M(B\times C, E\times F, G\times H, U\times W)+ O\left(\frac{q^{1/2}}{|A|}\right).\end{align*}
This implies that
\[N(g,g,g,g)=M(g,g,g,g)+ O\left(2^{C\epsilon^{-8}}\frac{q^{1/2}}{|A|}\right).\]
Using (\ref{mot1}), the above estimate, and (\ref{hai1}), with some big positive constant $C$,
one has
\begin{align*} N(S, S, S, S)&\ge N(g, g, g, g)-C\epsilon -C\frac{q^{1/8}}{|A|^{1/4}} \\
&\ge  M(g,g,g,g)-C 2^{C\epsilon^{-8}}\frac{q^{1/2}}{|A|}-C\epsilon -C\frac{q^{1/8}}{|A|^{1/4}}\\
&\ge   M(S,S,S,S) -C\epsilon -C 2^{C\epsilon^{-8}}\frac{q^{1/2}}{|A|}-C\epsilon -C\frac{q^{1/8}}{|A|^{1/4}}.\end{align*}
We know from (\ref{xxxx}) that 
\[M(S, S, S, S)\ge \left(\frac{|S|}{q^2}\right)^4.\]
Thus, we obtain that
\begin{align*} N(S, S, S, S)&\ge \left(\frac{|S|}{q^2}\right)^4-2C\epsilon -C 2^{C\epsilon^{-8}}\frac{q^{1/2}}{|A|} -C\frac{q^{1/8}}{|A|^{1/4}}\\
&\ge \frac{|S|^4}{q^8} -2C \epsilon -C 2^{C\epsilon^{-8}} q^{-1/6} -C q^{-1/24},\end{align*}
where we used the assumption that $|A|\ge q^{2/3}$ to get the last line.
Hence, $N(S,S,S,S) \ge {|S|^4}/{2q^8}$ under the following condition:
\begin{equation}\label{conditionK}
|S|\ge(4C\epsilon + 2C 2^{C\epsilon^{-8}} q^{-1/6} +2C q^{-1/24})^{1/4} q^2.
\end{equation}

Now, given $0<\delta\le 1$,  we take  $\epsilon=\frac{\delta^4}{8C}>0.$ 
In addition, notice that  we can choose a positive integer $q_0$ satisfying the following two conditions:
$$ {q_0}^{1/6} \ge \frac{ 2^{C\epsilon^{-8}}}{\epsilon} \quad \mbox{and} ~~  {q_0}^{1/24} \ge \frac{1}{\epsilon}.$$
Therefore,  if $q\ge q_0$, then  the right hand side of \eqref{conditionK} can be estimated as follows:
$$(4C\epsilon + 2C 2^{C\epsilon^{-8}} q^{-1/6} +2C q^{-1/24})^{1/4} q^2  \le ( 4C\epsilon + 2C\epsilon + 2C \epsilon)^{1/4}q^2=\delta q^2. $$
Combining this estimate with \eqref{conditionK},  it follows that if $|S|\ge \delta q^2,$ then $N(S,S,S,S)\ge \frac{|S|^4}{2q^8}$.

It follows from the definition of $N(S, S, S, S)$ that the number of rectangles in $S$ with side-lengths in $A$, one side parallel to the line $y=0$, and one side parallel to the line $x=0$ is at least $q^4\cdot N(S, S, S, S)\cdot \frac{|A|^2}{q^2}$. Taking the sum over all directions that can be rotated to the line $y=0$, the theorem follows for the case of arbitrary finite fields. $\square$
\section{Proof of Theorem \ref{thm33333}}
By modifying slightly the proof of Theorem \ref{thm2}, one can obtain the following result which will be used in the proof of Theorem \ref{thm33333}.
\begin{theorem}\label{bosung}
For any $0<\delta\le 1$, there exists an integer $q_0=q_0(\delta)$ with the following property: if $q\ge q_0$ and $A$ is a multiplicative subgroup of $\mathbb{F}^*_q$ with $|A|\ge q^{2/3}$, then  any set $S\subset \mathbb{F}_q^2$ with $|S|\ge \delta q^2$ contains at least $\gg \frac{|S|^2|A|}{q}$ pairs of points $\left((a, b), (a, c) \right)\in S\times S$ such that $b-c\in A$.
\end{theorem}
\begin{proof}[Proof of Theorem \ref{thm33333}]

Since $|S|\gg q^{3/2}$, it has been proved in \cite{IR-05} that the number of pairs of points of length $t$ in $S$ is $(1+o(1))|S|^2/q$. This means that there exists a direction $u$ such that there are at least $(1+o(1))|S|^2/q^2$ pairs $(x, y)\in S\times S$ such that $x-y=u$. Since $t$ is a square, by a rotation, we can assume that $u$ is parallel to the line $y=0$. Let $F=\{(x, y)\in S\times S \colon x-y=u\}$. 
Then we have $|F|\gg |S|^2/q^2$. Define $F':=\{x\colon \exists y, ~(x, y)\in F\}$. It is clear that $|F'|=|F|\sim \delta^2q^2$. 
We now apply Theorem \ref{bosung} to have at least $\gg \frac{|F'|^2|A|}{q}$ pairs of points of the form $(a, b)$ and $(a, d)$ in $F'$ such that $b-d\in A$. 
This means that there are at least $\gg \frac{|S|^4|A|}{q^5}$ rectangles in $S$ with one side length $t$ and the other side length in $A$. 
\end{proof}

\section{Some Remarks}

For prime fields, let us recall the following result due to Bourgain, Glibichuk, and Konyagin \cite{BGK}. 
\begin{theorem} \label{smallsetdecay} Let $A$ be a multiplicative subgroup of $\mathbb F_q^*$, where $q$ is prime. Then if $|A|\ge q^\alpha$ for $\alpha>0$, there exists $\beta>0$ depending only on $\alpha$ such that 
$$ \max_{a\in \mathbb F_q^*} \left|\sum_{x\in A} \chi(ax)\right|\le |A| q^{-\beta}.$$
\end{theorem}
An explicit relation between $\alpha$ and $\beta$ in certain range can be found in \cite{DDD, Ga10, Sh1, Shk}, but we only need a non-trivial estimate on $A$ with  arbitrary small size.

For $A$ with $|A|\ge q^{\alpha}$, if \begin{equation}\label{eq909}q^\beta\to \infty ~\mathtt{as}~q\to \infty,\end{equation} then Theorem \ref{thm2} and Theorem \ref{thm33333} also hold with the same argument. We note that the condition (\ref{eq909}) will not be satisfied if the size of $A$ is too small, for example, $A=\{1, -1\}$. 

In the statement of Theorems \ref{thm2} and \ref{thm33333}, if we do not assume that the side-lengths are in a multiplicative subgroup, then the problem becomes much easier. More precisely, one can use an elementary argument to show that for any set $S\subset \mathbb{F}_q^2$ such that $|S|\gg q^{3/2}$, then the number of rectangles in $S$ is at least $\gg |S|^4q^{-3}$. 

Indeed, let $X:=\{(x, y, x^2+y^2)\colon (x, y)\in S\}$. Then we have $X$ is a subset on a paraboloid in $\mathbb{F}_q^3$. By a direct computation, we can check that there is a correspondence between each rectangle in $S$ and an energy quadruple $(a, b, c, d)\in X^4$ with $a+b=c+d$, and $a, b, c, d$ are distinct. 

By the Cauchy-Schwartz inequality, the number of such energy quadruples is at least $\frac{|X|^4}{q^3}-2|X|^2$, where we used the fact that $|X+X|\le q^3$. Since $|X|=|S|$ and $|S|\gg q^{3/2}$, the number of rectangles in $S$ will be at least $\gg \frac{|S|^4}{q^3}$.

\section{Acknowledgments}
We would like to thank Neil Lyall,  Igor Shparlinski, and Ilya Shkredov for useful discussions and comments.

Doowon Koh was supported by the National Research Foundation of Korea (NRF) grant funded by the Korea government (MIST) (No. NRF-2018R1D1A1B07044469). Thang Pham was supported by Swiss National Science Foundation grant P4P4P2-191067. Chun-Yen Shen was supported in part by MOST, through grant 108-2628-M-002-010-MY4.

 \bibliographystyle{amsplain}

\end{document}